\documentclass[11pt,reqno]{amsart}
\usepackage{amssymb}
\usepackage{graphicx}
\usepackage{url}

{\theoremstyle{plain}
 \newtheorem{theorem}{Theorem}
 \newtheorem{corollary}{Corollary}

\newtheorem{lemma}{Lemma}
}

\begin{document}

\title{Somewhat smooth numbers in short intervals}

\begin{abstract}
We use exponent pairs to establish the existence of many $x^a$-smooth numbers in short intervals
$[x-x^b,x]$, when $a>1/2$.
In particular, $b=1-a-a(1-a)^3$ is admissible.
Assuming the exponent-pairs conjecture, one can take $b=(1-a)/2+\epsilon$. 
As an application, we show that $[x-x^{0.4872},x]$ contains many practical numbers when $x$ is large. 
\end{abstract}

\author{Andreas Weingartner}
\address{ 
Department of Mathematics,
351 West University Boulevard,
 Southern Utah University,
Cedar City, Utah 84720, USA}
\email{weingartner@suu.edu}
\date{\today}
\maketitle

\section{Introduction}

We say that a natural number $n$ is $y$-smooth if all of its prime factors are $\le y$. 
Let $\Psi(x,y)$ be the number of such $n\le x$. 
Improving on many earlier efforts by a number of different authors,
Matom\"aki and Radziwi\l\l \cite{MR} established the existence of many $x^\epsilon$-smooth numbers 
in intervals of the form $[x,x+c(\epsilon) \sqrt{x}]$, for every $\epsilon>0$. 
Harman \cite{Har} showed that intervals around $x$ of length $x^{0.45...}$ contain many $x^{0.27...}$-smooth numbers.  

We are interested in the existence of $x^a$-smooth numbers in much shorter intervals, when
$a> 1/2$. More precisely, given $a \in (1/2, 1)$, how small can we take $b$ such that  
\begin{equation*}\label{qeq}
\Psi(x,x^a)-\Psi(x-x^b,x^a) \gg  x^{b-\epsilon}
\end{equation*}
for every $\epsilon>0$?
In that direction, Friedlander and Lagarias \cite{FL} showed that there exists a constant $c>0$ such that 
$b=1-a - c a(1-a)^3$ is admissible, even with $\epsilon=0$, but without providing any numerical estimate for $c$. 
We will use exponent pairs (see \cite{GK}) to find explicit values of $b<1-a$. 
In particular, $b=1-a-a(1-a)^3$ is admissible for every $a \in  (1/2,1)$.

Let $\psi(x)=x-\lfloor x\rfloor -1/2 = \{x\}-1/2$. 
The method used by Friedlander and Lagarias \cite{FL} starts with Chebyshev's identity and 
requires estimates for sums of $\psi(x/p) \log p$, where $p$ runs over primes. 
Our approach involves sums of $\psi(x/n)$ over all integers $n$ from an interval.
We use the estimate
\begin{equation}\label{UB}
\sum_{N\le n \le 2N} \psi(x/n) \ll \min\bigl(x^\theta , x^{k/(k+1)} N^{(l-k)/(k+1)} \bigr)  \qquad (1\le N \le \sqrt{x}),
\end{equation}
where $(k,l)$ is any exponent pair. The two most recent records for $\theta$ are  $\theta=\frac{131}{416}+\epsilon =0.3149...$ by Huxley \cite[Thm. 4]{Hux} and $\theta=\frac{517}{1648}+\epsilon=0.3137...$ by Bourgain and Watt \cite[Eq. (7.4)]{BW}.
For the second estimate in \eqref{UB}, see Graham and Kolesnik \cite[Lemma 4.3]{GK}. 

Let $\nu=2.9882...$ be the minimum value of $(2^u-1)/(u-1)$ for $u>1$. 
\begin{theorem}\label{thm1}
Let $(k,l)$ be an exponent pair and $\theta$ as in \eqref{UB}.
There is a constant $K$ such that 
$$\Psi(x,y)-\Psi(x-z,y) \gg \frac{z}{(\log x)^\nu},$$ 
provided $x \ge y \ge \sqrt{2x}$ and $x\ge z\ge K \min\bigl(x^\theta,x^{l/(k+1)}y^{(k-l)/(k+1)}\bigr)$.
\end{theorem}

Define
\begin{equation}\label{bdef}
b=b(a,k,l) = \frac{l+a(k-l)}{k+1}.
\end{equation}

\begin{corollary}\label{cor1}
Let $(k,l)$ be an exponent pair, $\theta$ as in \eqref{UB} and $1/2< a\le   1$.
There is a constant $K$ such that for $x\ge z\ge K x^{\min(\theta ,b)}$,
$$\Psi(x,x^a)-\Psi(x-z,x^a) \gg \frac{z}{(\log x)^\nu}.$$ 
If $a=1/2$, the conclusion holds if $x^a$ is 
replaced by $\sqrt{2x}$.
\end{corollary}

Starting with the exponent pair $(\kappa,\lambda)=(13/84+\epsilon,55/84+\epsilon)$ 
of Bourgain \cite[Thm. 6]{Bourgain}, and possibly applying van der Corput's processes $A$ or $B$,
we find a sequence of linear functions in $a$, shown in Table \ref{table1}.
When $a$ is close to $1/2$, then $\theta$ is smaller than any $b$ obtained from known exponent pairs. 
When $a$ is close to $1$, we rely on exponent pairs $(k,l)$ with small $k$. 
Heath-Brown \cite[Thm. 2]{HB} found that for integers $m\ge 3$ and every $\epsilon >0$,
\begin{equation}\label{hbeq}
k_m=\frac{2}{(m-1)^2 (m+2)}, \quad l_m= 1-\frac{3m-2}{m(m-1)(m+2)} +\epsilon
\end{equation}
is an exponent pair. This enables us to prove the following result. 

\begin{corollary}\label{cor2}
For each $a \in [1/2,1)$, the conclusion of Corollary \ref{cor1} holds for some $b<1-a - a(1-a)^3  - 4.32 \, a(1-a)^5$.
\end{corollary}

The value of $a$, for which $b(a,k_m,l_m)=b(a,k_{m+1},l_{m+1})$,  is given by
$$
a_m := 1 -\frac{1}{m} + \frac{2-m^{-1}}{m^3+m^2+2m-1}
$$
If $a>0.796...$ and $a\in [a_{m-1},a_m]$, then $b$ is minimized by $b(a,k_m,l_m)$. 
This yields slightly smaller values of $b$ than Corollary \ref{cor2}.

\smallskip
\begin{table}[h]
     \begin{tabular}{ | c |  c |  c |  c | c | }
    \hline
     $b$   & Interval for $a$ &  Exponent Pair \\ \hline   
     $517/1648 +\epsilon$ & $[0.500..., 0.579...]$ &   \\ \hline
    $ (110-55a)/249+\epsilon$ & $[0.579..., 0.590...]$ &  $BA(\kappa,\lambda)$ \\ \hline
     $ (55-42a)/97+\epsilon$ & $[0.590..., 0.701...]$ &  $(\kappa,\lambda)$ \\ \hline
        $ (152-139a)/207+\epsilon$ & $[0.701..., 0.766...]$ &  $A(\kappa,\lambda)$ \\ \hline
          $ (359-346a)/427+\epsilon$ & $[0.766..., 0.796...]$ &  $AA(\kappa,\lambda)$ \\ \hline
          $  b(a,k_m,l_m)$ & $  [a_{m-1},a_m], \ m\ge 5$ &  $(k_m,l_m)$ \\ \hline
     \end{tabular}        
   \medskip
     \caption{Admissible values of $b$, depending on $a$.}\label{table1}
\end{table}

The values $a=1-1/m$, where $m\ge 2$  is an integer, may be of particular interest.
Here we have $a_{m-1}< a=1-1/m < a_m$ and
$$ b = b(1-1/m,k_m,l_m) = \frac{(m-1) \left(m^3+m^2-3 m+2\right)}{m^2 \left(m^3-3 m+4\right)} +\epsilon.
$$ 

The exponent-pairs conjecture states that $(k,l)=(\epsilon, 1/2+\epsilon)$ is an exponent pair for every $\epsilon>0$. 
\begin{corollary}\label{cor3}
 If $(\epsilon, 1/2+\epsilon)$ is an exponent pair, then the conclusion of Corollary \ref{cor1} holds with $b=(1-a)/2+\epsilon$
 for each $a \in [1/2,1]$.
\end{corollary}

 \begin{figure}[h]\label{fig1}
\begin{center}
\includegraphics[height=70mm,width=112mm]{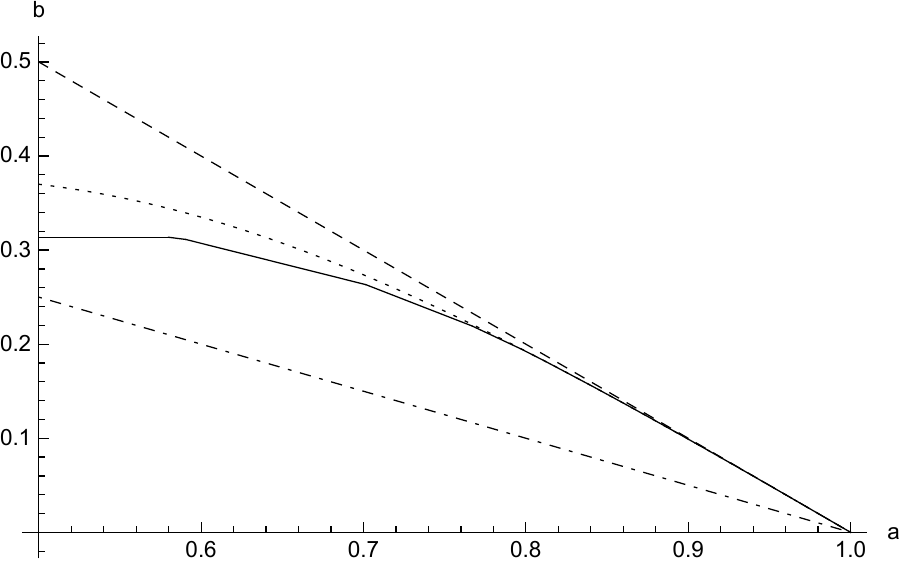}
\caption{Admissible values of $b$ based on Table \ref{table1} (solid);
$b=1-a-a(1-a)^3  - 4.32 \, a(1-a)^5$ (dotted) from Cor.\ \ref{cor2};  
$b=1-a$ (dashed) from the exponent pair $(k,l)=(0,1)$; 
 $b=\frac{1}{2}(1-a)$ (dot-dashed) from the exponent-pairs conjecture.}
\label{figure2}
\end{center}
\end{figure}

If one is only concerned with the existence of a single $y$-smooth number 
in short intervals, then a construction due to Friedlander and Lagarias \cite{FL} 
(consider integers of the form $m^2-h^2=(m-h)(m+h)$, where $m=\lceil \sqrt{x}\rceil$ 
and $h=0,1,2,\ldots$) and an easy exercise (aided by a computer to deal with small values of $x$) lead to
the explicit estimate
$$
\Psi\bigl(x,\sqrt{2x}\bigr)  - \Psi\bigl(x-3 x^{1/4},\sqrt{2x}\bigr)  \ge 1 \qquad (x\ge 1).
$$
From Table \ref{table1}, we find that our intervals are wider than $3x^{1/4}$ 
when $a< 401/556=0.721...$, but are shorter when $a>401/556$. 

\smallskip

\section{Proofs}

Let $\tau(n)$ be the number of positive divisors of $n$.
The following estimate is a special case of Theorem 2 of Shiu \cite{Shiu}.
\begin{lemma}\label{ShiuLem}
Let $\epsilon>0$ and $u \in \mathbb{R}$ be fixed. For $x \ge 2$ and $x^\epsilon \le z \le x$, we have
$$
\sum_{x-z\le n \le x} (\tau (n))^u \ll z (\log x)^{2^u -1}.
$$
\end{lemma}

\begin{proof}[Proof of Theorem \ref{thm1}]
Let $P(n)$ denote the largest prime factor of $n$.
Note that the result holds if $z>x/2$, so we may assume $z\le x/2$. 
Define
\begin{equation}\label{Sdef}
S := \sum_{x/y\le d \le 2x/y} \sum_{x-z<n\le x \atop n \equiv 0 \bmod d} 1.
\end{equation}
We have
\begin{equation*}
\begin{split}
S&  =\sum_{x/y\le d \le 2 x/y} \left( \lfloor x/d \rfloor -  \lfloor (x-z)/d \rfloor \right) \\
& =\sum_{x/y\le d \le 2 x/y} z/d - \sum_{x/y\le d \le 2x/y}  \psi(x/d) + \sum_{x/y\le d \le 2x/y}  \psi((x-z)/d) \\
&\ge z/3+O\bigl(\min\bigl(x^\theta,x^{l/(k+1)}y^{(k-l)/(k+1)}\bigr)\bigr)  \\
& \ge z/4,
\end{split}
\end{equation*}
by \eqref{UB} and the assumptions of Theorem \ref{thm1}.
 
Note that $y\ge \sqrt{2x}$ implies $2x/y \le y$. 
Every $n$ counted in the inner sum of \eqref{Sdef} has a divisor $d \in [x/y,2x/y]\subseteq [x/y,y]$. 
Since $d\le y$ and $n/d\le x/(x/y)=y$, we have $P(n)\le y$, i.e. $n$ is $y$-smooth.  
Moreover, each $n$ is counted at most $\tau(n)$ times, once for each divisor $d$ of $n$ with $d\in [x/y,2x/y]$. 
Thus,
$$
S \le \sum_{x-z<n\le x \atop P(n)\le y} \tau(n).
$$
For real numbers $t , u>1$ with $1/t +1/u=1$, H\"older's inequality yields
\begin{equation*}
\begin{split}
S & \le \Biggl(\sum_{x-z<n\le x \atop P(n)\le y} 1\Biggr)^{1/t}  \Biggl(\sum_{x-z<n\le x }\tau(n)^u \Biggr)^{1/u} \\
& \ll \left(\Psi(x,y)-\Psi(x-z,y)\right)^{1/t}  z^{1/u} (\log x)^{(2^u-1)/u},
\end{split}
\end{equation*}
by Lemma \ref{ShiuLem}.
Since $S\ge z/4$, we get
$$
\Psi(x,y)-\Psi(x-z,y) \gg \frac{z}{(\log x)^{(2^u-1)/(u-1)}}.
$$
The last exponent has a minimum value of $\nu=2.9882...$ at $u=2.1080...$
\end{proof}

\smallskip

\begin{proof}[Proof of Corollary \ref{cor2}]
For $m\ge 3$ and $a\in [a_{m-1},a_m]$, we want to show that $b(a,k_m,l_m) < f(a)$, where $f(a)=1-a-a(1-a)^3-4.32a(1-a)^5$. 
Since $f''(a)<0$ for $1/2<a<1$ and $b(a,k_m,l_m)$ is a linear function in $a$ for each $m$,
it suffices to verify the inequality at the endpoints $a=a_m$. That is, we need to show that $b(a_m,k_m,l_m) < f(a_m)$
for $m\ge 2$. We find that $f(a_m)-b(a_m,k_m,l_m) $ is a rational function in $m$ that is positive for every $m\ge 1$. 
This proves the claim for $a\ge a_2 = 3/5$. 
If $1/2 \le a < 3/5$, the result follows from Table \ref{table1}. 
\end{proof}

\section{Application to practical numbers}

Let $\mathcal{A}$ be the set of positive integers containing $n=1$ and all those $n \ge 2$ with prime factorization $n=p_1^{\alpha_1} \cdots p_k^{\alpha_k}$, $p_1<p_2<\ldots < p_k$, which satisfy $p_1=2$ and 
\begin{equation*}
p_{i} \le  p_1^{\alpha_1} \cdots p_{i-1}^{\alpha_{i-1}} \qquad (2\le i \le k).
\end{equation*}

The significance of the set $\mathcal{A}$ is that it is a subset of several notable integer sequences,
including the practical numbers (i.e. integers $n$ such that every natural number $m\le n$ can be expressed as a sum of distinct positive divisors of $n$ \cite{PW, Saias,Ten86,PDD}), the $t$-dense numbers, for every $t\ge 2$, (i.e. the ratios 
of consecutive divisors of $n$ are at most $t$, \cite{PW, Saias,Ten86, PDD}), and the $\varphi$-practical numbers (i.e. $x^n-1$ has a divisor in $\mathbb{Z}[x]$ of every degree up to $n$, \cite{PTW}).

Let $\nu=2.9882...$ be as in Theorem \ref{thm1}, $C=(1-e^{-\gamma})^{-1}=2.280...$, where $\gamma=0.5772...$ is
Euler's constant, and 
$$\mu_0 := 2\nu+2+C\log 2 = 9.5569...$$ 
\begin{theorem}\label{thm2}
 Let $(k,l)$ be an exponent pair, $\displaystyle \beta =\frac{5k+l+2}{6(k+1)}$ and $\mu >\mu_0$.
There exists a constant $K$ such that for $x\ge z \ge K x^\beta$, the interval $[x-z,x]$ contains 
$\gg z (\log x)^{-\mu}$  members of $\mathcal{A}$. 
\end{theorem}

The exponent pair $(k,l)=(13/194+\epsilon,76/97+\epsilon)=A(\kappa,\lambda)$ yields:
\begin{corollary}\label{cor4}
For every $\beta >605/1242 = 0.4871...$  and $\mu >\mu_0$, the conclusion of Theorem \ref{thm2} holds. 
Assuming the exponent-pairs conjecture,  it holds for every $\beta > 5/12=0.4166...$. 
\end{corollary}

\begin{corollary}\label{cor5}
The interval $[x-x^{0.4872},x]$ contains at least $x^{0.4872}(\log x)^{-9.557}$ members of 
$\mathcal{A}$, for all sufficiently large $x$.
\end{corollary}

A quick search on a computer suggests that Corollary \ref{cor5} probably holds for all $x\ge 504$. 

It is clear that Theorem \ref{thm2} and its corollaries  remain valid if $\mathcal{A}$ is replaced by any superset of $\mathcal{A}$.
In the case of practical numbers, Corollary \ref{cor5} improves on two earlier results: 
Hausman and Shapiro \cite{HS} found that the interval $[x^2, (x+1)^2]$ contains a practical  number for every $x\ge 1$,
in analogy with Legendre's conjecture for primes. 
 Melfi \cite[Thm. 9]{Mel95} sharpened this by showing that the 
 interval $[x,x+K\sqrt{x/\log\log x}]$ contains a practical number for all large $x$ and some constant $K$.
  
 Granville \cite[Conj. 4.4.2]{Gran} states the conjecture that for every fixed $\epsilon >0$,
the interval $[x-x^\epsilon,x]$ contains a $x^\epsilon$-smooth number for all $x\ge x_0(\epsilon)$. 
Pomerance \cite{PTalk} points out that this would imply the existence of a practical number 
(or member of $\mathcal{A}$) in every interval $[x-x^\epsilon,x]$ for large $x$. 

The following observation follows at once from the definition of the set $\mathcal{A}$.

\begin{lemma}\label{ML}
If $n \in \mathcal{A}$ and $P(m)\le n$, then $mn \in \mathcal{A}$. 
\end{lemma}

\begin{proof}[Proof of Theorem \ref{thm2}]
If $z> x/2$, the result follows from Theorem 1.2 of \cite{PDD}, so we may assume $z\le x/2$. 
Let $a=3/4$. We have $\displaystyle b=\frac{3k+l}{4(k+1)}>0$, according to \eqref{bdef}, and $\beta = 1/3+(2/3)b>1/3$. 

Theorem 1.2 of \cite{PDD} shows that the number of $n \in  \mathcal{A} \cap  (2x^{1/3}, 3x^{1/3}]$ is $ \sim c x^{1/3}/\log x$
for some positive constant $c$. 
Let $\epsilon>0$ and $C=(1-e^{-\gamma})^{-1}=2.280...$. 
By Corollary 1 of \cite{OMDD}, the number of these $n$ with $\Omega(n)>(C+\epsilon)\log\log n$ is $o(x^{1/3}/\log x)$, so we may exclude such $n$ and assume
$\Omega(n)\le (C+\epsilon)\log\log n$. 

Since $n \in (2x^{1/3}, 3x^{1/3}]$, the condition $z\ge 3K x^\beta$ implies $z/n \ge K (x/n)^b$. 
By Corollary \ref{cor1}, for each of these $n$, 
the interval $I_n:=[x/n - z/n, x/n]$
contains $\gg (z/n)(\log x/n)^{-\nu} \gg z x^{-1/3}(\log x)^{-\nu}$ integers $m$ that are $(x/n)^{3/4}$-smooth. 
Note that $mn \in [x-z,x]$ for $m\in I_n$.

We will show that for each of these pairs $(n,m)$ as described above, we have $mn \in \mathcal{A}$. 
Let $p = P(m)$. Since $n\ge 2x^{1/3}$, $p \le (x/n)^{3/4} \le x^{1/2}2^{-3/4}$. 
If $p\le x^{1/3}$, then $mn \in \mathcal{A}$, by Lemma \ref{ML}.
If $p> x^{1/3}$, write $m=pr$ and note that $r=m/p \le x/(np) < x^{1/3}$. 
Thus, $rn \in \mathcal{A}$ by Lemma \ref{ML}. 
Since $p\le x^{1/2}2^{-3/4}$, we have 
$ p^2 \le x 2^{-3/2}  < mn = pr n$ and hence $p < rn$.
Thus, $mn =prn \in \mathcal{A}$ also holds in this case, by Lemma \ref{ML}. 

The number of pairs $(m,n)$ is $\gg z(\log x)^{-1-\nu}$,
but several pairs may lead to the same product $mn$. 
We have $\tau(n)\le 2^{\Omega(n)} \le (\log x)^{C \log 2 + \epsilon}$.
By Lemma \ref{ShiuLem}, we have $\sum_{m\in I_n} \tau(m) \ll (z/n)\log x$.
Since the number of $m\in I_n$ that are $(x/n)^{3/4}$-smooth is $ \gg (z/n) (\log x)^{-\nu}$,
we have $\tau(m) \ll (\log x)^{\nu+1}$ for a positive proportion of them.
Thus, we may assume $\tau(m) \ll (\log x)^{\nu+1}$, and therefore $\tau(mn)\ll (\log x)^{\nu+1+C\log 2 +\epsilon}$.
It follows that the number of distinct products $mn$ is 
$$
\gg  \frac{z (\log x)^{-1-\nu}}{ (\log x)^{\nu+1+C\log 2 +\epsilon}}
=\frac{z}{(\log x)^{\mu_0 + \epsilon}}.
$$
\end{proof}

\end{document}